\documentclass{article}
\usepackage{amssymb}
\usepackage{amsmath}
\usepackage[normalem]{ulem}
\usepackage[utf8]{inputenc}
\usepackage{graphicx}
\usepackage{color}

\newtheorem{theorem}{Theorem}[section]

\newtheorem{proposition}[theorem]{Proposition}
\newtheorem{corollary}[theorem]{Corollary}
\newtheorem{definition}{Definition}[section]
\newtheorem{preexample}{Example}[section]
\newenvironment{example}{\begin{preexample}}{\end{preexample}}
\newtheorem{preremark}{Remark}
\newenvironment{remark}{\begin{preremark}\rm}{\end{preremark}}

\newenvironment{proof}
{{\bf Proof:}}
{\qquad \hspace*{\fill} $\Box$}%

\title{Vertical Control Systems on Tangent Bundles and Fiberwise Controllability}

\author{S. N. Stelmastchuk\\
Universidade Federal do Paran\'a, Jandaia do Sul, Brazil}

\begin{document}
	
\maketitle
	
  \begin{abstract}
    We study control systems on the tangent bundle of a smooth manifold induced by vertical lifts of vector fields. The Vertical dynamics acts exclusively along the fibers, leaving the base point unchanged and reducing the system to a linear control problem on each tangent space, for which we obtain explicit solutions and characterize reachable sets, showing that fiberwise controllability is equivalent to a rank condition on the original vector fields. We then consider lifted systems combining complete drift and vertical controls, where the base trajectory is fixed by the drift and the control acts on tangent directions. For these systems, we derive explicit solutions and a complete characterization of reachable sets via a transport operator, yielding a necessary and sufficient condition for fiberwise controllability in terms of transported vector fields, together with a Lie-algebraic sufficient criterion.
  \end{abstract}

	
\section{Introduction}

Let $M$ be a smooth manifold and consider its tangent bundle $TM$. In this paper, we investigate control systems on $TM$ arising from vertical lift of vector fields.

Recall that a geometric control system on $M$ is given by
\begin{equation}\label{eq:controlintroduction}
  \dot{x} = X_0(x) + \sum_{i=1}^m u_i X_i(x), \qquad x \in M,
\end{equation}
where $X_0, X_1, \ldots, X_m \in \mathfrak{X}(M)$ are smooth vector fields and $u = (u_1, \ldots, u_m)$ is an admissible control. The study of such systems is classical in geometric control theory; see, for instance, 
\cite{agrachev, sontag, colonius, jurdjevic}.

A natural way to enrich the analysis of \eqref{eq:controlintroduction} is to lift it to the tangent bundle, where one can describe not only the evolution of trajectories but also the behavior of velocities along them. The purpose of such a lift is to analyze the evolution of tangent vectors along controlled trajectories.

An initial study in this direction was carried out by the author in \cite{stelmastchuk}, where control systems induced by complete lifts were investigated. In that framework, the lifted dynamics follows the evolution of trajectories on $M$, and several structural properties were derived using the vector bundle structure of $TM$.

The present paper extends this line of research by shifting the focus to vertical lifts, which lead to a fundamentally different class of control systems. In contrast to the complete lift, vertical dynamics does not affect the base point and induces a purely fiberwise evolution. The notion of vertical structures also appears naturally in the study of mechanical systems; see, for instance, \cite{lewis}, and plays a relevant role in optimal control theory.

Vertical lifts provide a canonical mechanism to encode control actions that act exclusively on tangent directions, leaving the base point unchanged. As a consequence, the resulting dynamics is purely fiberwise: trajectories remain in a fixed tangent space $T_{x_0}M$, and the system reduces to a linear control system on each fiber. This structure allows for a transparent characterization of controllability, entirely determined by the span of the original vector fields at the base point.

We formalize this setting by introducing the notion of a vertical control system and, in particular, the vertical lift of \eqref{eq:controlintroduction}, given by
\begin{equation}\label{eq:verticalintroduction}
  \dot v = X_0^v(v) + \sum_{i=1}^m u_i X_i^v(v), \qquad v \in TM.
\end{equation}
For this system, we obtain an explicit solution (Theorem~\ref{thm:explicit_vertical_solution}) and provide a complete description of its reachable set. In particular, we show that \eqref{eq:verticalintroduction} is controllable on the fiber $T_{x_0}M$ if 
\[
  \operatorname{span}\{X_1(x_0),\dots,X_m(x_0)\} = T_{x_0}M.
\]
We emphasize that controllability is inherently fiberwise: the system is never controllable on the whole tangent bundle $TM$.

Having established the vertical framework, we then consider a more general lifted system combining vertical and complete dynamics:
\begin{equation}\label{eq:dynamicalintroduction}
  \dot v(t) = Y^c(v(t)) + \sum_{i=1}^m u_i(t)\, X_i^v(v(t)), \qquad v(t) \in TM.
\end{equation}
This system admits a natural interpretation: its projection onto $M$ follows the dynamics $\dot{x}(t) = Y(x(t))$, while the vertical component allows one to modify the velocity along the trajectory. In this sense, the control acts on the tangent directions without altering the base evolution, making it possible to control the velocity profile along a prescribed path.

Although the resulting dynamics is structurally simple, the vertical framework provides a clear and explicit bridge between geometric control systems on $M$ and linear controllability properties on the fibers, allowing a precise characterization of reachable sets and their dependence on the underlying vector fields.

For system \eqref{eq:dynamicalintroduction}, we derive an explicit representation of solutions, characterize reachable sets through a transport operator, and obtain a Lie-algebraic sufficient condition for fiberwise controllability in terms of the iterated brackets $\operatorname{ad}_Y^kX_i$.

This paper is organized as follows. Section 2 introduces the geometric structure of vertical subspaces on the tangent bundle. Section 3 develops the theory of vertical control systems and establishes their main controllability properties. Finally, Section 4 is devoted to the analysis of the combined lifted system \eqref{eq:dynamicalintroduction}.

\section{Vertical subspace on the tangent bundle}

In this paper, we closely follow the classical framework introduced by Yano and Ishihara \cite{kentaro} concerning lifts of geometric objects to the tangent bundle. For the reader’s convenience, we briefly recall the main definitions and properties that will be used throughout the paper.

Let $M$ be a smooth manifold and let $\pi: TM \to M$ denote its tangent bundle with the canonical projection. Given a smooth function $f \in C^\infty(M)$, its differential $df$ naturally defines a function on $TM$, which we denote by $\mathfrak{i}(df)$. In local coordinates $(U, x)$ on $M$, with induced coordinates $(\pi^{-1}(U), x, y)$ on $TM$, this function is expressed as
\[
  \mathfrak{i}(df)(x,y) = \frac{\partial f}{\partial x_i}(x)\, y_i.
\]

The \emph{complete lift} of $f$ is the function $f^c \in C^\infty(TM)$ defined by
\[
  f^c = \mathfrak{i}(df),
\]
while the \emph{vertical lift} of $f$ is given by
\[
  f^v = f \circ \pi.
\]

Let $X \in \mathfrak{X}(M)$ be a smooth vector field. The \emph{complete lift} $X^c \in \mathfrak{X}(TM)$ is defined by the identity
\[
  X^c(f^c) = (Xf)^c, \qquad \text{for all } f \in C^\infty(M),
\]
and the \emph{vertical lift} $X^v \in \mathfrak{X}(TM)$ is characterized by
\[
  X^v(f^v) = 0, \qquad \text{for all } f \in C^\infty(M).
\]

These lifted vector fields have clear geometric interpretations. The vertical lift $X^v$ describes variations along the fibers of $TM$: for any $w \in T_xM$,
\[
  X^v(w) = \left.\frac{d}{dt}\right|_{t=0} \big(w + t X(x)\big).
\]
On the other hand, the complete lift $X^c$ encodes the differential of the flow of $X$. If $\phi_t$ denotes the flow of $X$, then
\[
  X^c(w) = \left.\frac{d}{dt}\right|_{t=0} (d\phi_t)_x(w).
\]

In local coordinates $(x_i)$ on $M$, if
\[
  X = X^i \frac{\partial}{\partial x_i},
\]
then its vertical and complete lifts are given, respectively, by
\[
  X^v = X^i \frac{\partial}{\partial y_i}, 
  \qquad 
  X^c = X^i \frac{\partial}{\partial x_i} + y_j \frac{\partial X^i}{\partial x_j} \frac{\partial}{\partial y_i}.
\]

We now recall some fundamental algebraic properties of these lifts. For all $X,Y \in \mathfrak{X}(M)$ and $\alpha,\beta \in \mathbb{R}$,
\begin{equation}\label{propertiesI}
  (\alpha X + \beta Y)^v = \alpha X^v + \beta Y^v, 
  \ \ 
  (\alpha X + \beta Y)^c = \alpha X^c + \beta Y^c,
  \ \ 
  X^c(f^v) = (Xf)^v,
\end{equation}
and
\begin{equation}\label{propertiesII}
  [X^v,Y^v] = 0, 
  \qquad 
  [X^c,Y^v] = [X,Y]^v, 
  \qquad 
  [X^c,Y^c] = [X,Y]^c.
\end{equation}

Finally, we observe that the vertical lift $X^v$ projects to the zero vector field on $M$, whereas the complete lift $X^c$ is $\pi$-related to $X$, that is,
\[
  d\pi \circ X^c = X \circ \pi.
\]

We now introduce the notion of vertical vector fields on $TM$ in an intrinsic way.

\begin{definition}
  Let $\pi: TM \to M$ be the canonical projection. For each $v_x \in TM$, the \emph{vertical subspace} at $v_x$ is defined by
  \[
    \mathcal{V}_{v_x}TM := \ker\big(d\pi_{v_x}\big) \subset T_{v_x}(TM).
  \]
  A vector $W \in T_{v_x}(TM)$ is said to be \emph{vertical} if
  \[
    W \in \mathcal{V}_{v_x}TM.
  \]
  A vector field $V \in \mathfrak{X}(TM)$ is called a \emph{vertical vector field} if
  \[
    V(v_x) \in \mathcal{V}_{v_x}TM \quad \text{for all } v_x \in TM.
  \]
\end{definition}

Geometrically, the vertical subspace $\mathcal{V}_{v_x}TM$ consists of all tangent vectors at $v_x$ that are tangent to the fiber $\pi^{-1}(x)$. In particular, the fiber $\pi^{-1}(x)$ is an embedded submanifold of $TM$, and
\[
  \mathcal{V}_{v_x}TM = T_{v_x}\big(\pi^{-1}(x)\big).
\]

This characterization allows us to interpret vertical vectors as infinitesimal variations that affect only the velocity component, while keeping the base point fixed.

By complete analogy, when considering the tangent bundle of $TM$, namely $TTM$, the notion of vertical subspace is defined with respect to the canonical projection
\[
  \pi_{TM} : TTM \to TM.
\]
For each $W \in TTM$, the vertical subspace at $W$ is given by
\[
  \mathcal{V}_W TTM := \ker\big(d\pi_{TM}(W)\big) \subset T_W(TTM).
\]
Thus, vertical vectors in $TTM$ represent infinitesimal variations that preserve the base point in $TM$, i.e., they are tangent to the fibers of the bundle $TTM \to TM$.

\begin{proposition}\label{prop:vertical_equivalence}
  Let $\pi: TM \to M$ be the canonical projection and fix $v_x \in TM$. Then
  \[
    \mathcal{V}_{v_x}TM = T_{v_x}(T_xM).
  \]
  In particular, for $W \in T_{v_x}(TM)$, the following are equivalent:
  \begin{enumerate}
    \item[(i)] $W$ is tangent to the fiber $T_xM$ at $v_x$, i.e., $W \in T_{v_x}(T_xM)$;
    \item[(ii)] $d\pi_{v_x}(W) = 0$.
  \end{enumerate}
\end{proposition}

\begin{proof}
  We first show that $T_{v_x}(T_xM) \subset \ker(d\pi_{v_x})$. Let $W \in T_{v_x}(T_xM)$. Then there exists a smooth curve $\gamma:(-\varepsilon,\varepsilon)\to T_xM$ such that $\gamma(0)=v_x$ and
  \[
    W = \frac{d}{dt}\Big|_{t=0} \gamma(t).
  \]
  Since $\pi(\gamma(t)) = x$ for all $t$, it follows that
  \[
    d\pi_{v_x}(W) = \frac{d}{dt}\Big|_{t=0} \pi(\gamma(t)) = 0,
  \]
  and therefore $W \in \ker(d\pi_{v_x})$.

  Conversely, let $W \in \ker(d\pi_{v_x})$ and let $\Gamma:(-\varepsilon,\varepsilon)\to TM$ be a smooth curve such that $\Gamma(0)=v_x$ and
  \[
    W = \frac{d}{dt}\Big|_{t=0} \Gamma(t).
  \]
  Define $\alpha(t):=\pi(\Gamma(t))$. Then
  \[
    \alpha'(0) = d\pi_{v_x}(W) = 0.
  \]
  In local coordinates $(x^i,y^i)$ on $TM$, write $\Gamma(t)=(x^i(t),y^i(t))$. The condition $\alpha'(0)=0$ implies that $x^{i\,\prime}(0)=0$ for all $i$. Hence, at $t=0$, the curve $\Gamma$ has no variation in the base variables, and its velocity is entirely determined by the fiber coordinates. This shows that $W$ is tangent to the fiber $T_xM$ at $v_x$, i.e., $W \in T_{v_x}(T_xM)$.

  Therefore, $\ker(d\pi_{v_x}) = T_{v_x}(T_xM)$, and the equivalence follows.
\end{proof}

As a direct consequence of Proposition \ref{prop:vertical_equivalence}, the vertical lift $X^v$ of any vector field $X \in \mathfrak{X}(M)$ is a vertical vector field on $TM$, since
\[
  X^v(v_x) \in \ker(d\pi_{v_x}) \quad \text{for all } v_x \in TM.
\]

\begin{example}
  Let $M = S^2$ and consider local spherical coordinates $(\theta,\varphi)$, with induced coordinates $(\theta,\varphi,y^\theta,y^\varphi)$ on $TS^2$. 

  Let $X = \frac{\partial}{\partial \varphi}$, the vector field generating rotations around the vertical axis. Its vertical lift is given by
  \[
    X^v = \frac{\partial}{\partial y^\varphi}.
  \]

  The corresponding flow acts only along the fiber:
  \[
    (\theta,\varphi,y^\theta,y^\varphi) \longmapsto (\theta,\varphi,y^\theta,y^\varphi + t),
  \]
  and therefore
  \[
    d\pi(X^v)=0, \qquad \pi \circ \Phi_t^{X^v} = \pi,
  \]
  where $\Phi_t^{X^v}$ is the flow of $X^v$. This shows that $X^v$ is a vertical vector field.
\end{example}

\begin{example}\label{ex:accel_s2}
  Let $M=S^2$ and consider local coordinates $(\theta,\varphi)$ with induced coordinates $(\theta,\varphi,y^\theta,y^\varphi)$ on $TS^2$. Define the vector field
  \[
    A = -\lambda\,y^\theta\,\frac{\partial}{\partial y^\theta} - \lambda\,y^\varphi\,\frac{\partial}{\partial y^\varphi}, \qquad \lambda>0.
  \]
  Since $d\pi\!\left(\frac{\partial}{\partial y^\theta}\right)=d\pi\!\left(\frac{\partial}{\partial y^\varphi}\right)=0$, it follows that $d\pi(A)=0$, and hence $A$ is vertical.

  The associated flow keeps the base point fixed and satisfies
  \[
    y^\theta(t)=e^{-\lambda t}y^\theta(0), \qquad y^\varphi(t)=e^{-\lambda t}y^\varphi(0),
  \]
  describing an isotropic exponential damping along the fibers.
\end{example}

\begin{remark}\label{rem:second_order}
  A second-order differential equation on $M$ can be written on $TM$ as
  \[
    \dot x^i = y^i, \qquad \dot y^i = a^i(x,y),
  \]
  which defines the vector field
  \[
    \mathcal{S} = y^i\frac{\partial}{\partial x^i} + a^i(x,y)\frac{\partial}{\partial y^i}.
  \]
  The term
  \[
    A = a^i(x,y)\frac{\partial}{\partial y^i}
  \]
  is a vertical vector field, since $d\pi(A)=0$.

  Thus, acceleration laws naturally induce vertical vector fields on $TM$. In particular, vertical lifts $X^v$ correspond to accelerations depending only on the base point, while general $a^i(x,y)$ allow velocity-dependent effects.
\end{remark}

\section{Vertical control systems}

We now introduce control systems on the tangent bundle whose dynamics are entirely constrained to the fibers.

\begin{definition}\label{def:vertical_control_system}
  Let $M$ be a smooth manifold and let $\pi: TM \to M$ be the canonical projection. A control system on $TM$ of the form
  \[
    \dot v = F(v,u), \qquad v \in TM, \quad u \in U \subset \mathbb{R}^m,
  \]
  is called a \emph{vertical control system} if
  \[
    F(v,u) \in \ker(d\pi_v) \quad \text{for all } v \in TM \text{ and } u \in U.
  \]
\end{definition}

In other words, a vertical control system generates trajectories that remain within each fiber, leaving the base point unchanged.

\begin{proposition}\label{prop:local_form_vertical}
  In induced coordinates $(x^i,y^i)$ on $TM$, a control system is vertical if and only if it can be written as
  \[
    \dot x^i = 0, \qquad \dot y^i = f^i(x,y,u), \quad i=1,\dots,n.
  \]
\end{proposition}

\begin{proof}
  In induced coordinates, the projection is given by $\pi(x,y)=x$, and its differential satisfies
  \[
    d\pi_{(x,y)}\!\left(\sum a^i \frac{\partial}{\partial x^i} + \sum b^i \frac{\partial}{\partial y^i}\right)
    =
    \sum a^i \frac{\partial}{\partial x^i}.
  \]
  Thus, $F(v,u) \in \ker(d\pi_v)$ if and only if its components in the $\partial/\partial x^i$ directions vanish, which is equivalent to $\dot x^i = 0$. The remaining components define $\dot y^i = f^i(x,y,u)$.
\end{proof}

We now present an example of a vertical control system on \(TS^2\).

\begin{example}
  Let $M=S^2$ and consider local spherical coordinates $(\theta,\varphi)$, with induced coordinates $(\theta,\varphi,y^\theta,y^\varphi)$ on $TS^2$. Define the control system
  \[
    \dot\theta = 0,\qquad \dot\varphi = 0,
  \]
  \[
    \dot y^\theta = -y^\theta + u_1,\qquad
    \dot y^\varphi = -\sin(\theta)\,y^\varphi + u_2,
  \]
  where $u=(u_1,u_2)\in\mathbb{R}^2$. Since the system has no components in the base directions, it is vertical. Equivalently,
  \[
    F(\theta,\varphi,y^\theta,y^\varphi,u)
    = (-y^\theta+u_1)\frac{\partial}{\partial y^\theta} + (-\sin(\theta)\,y^\varphi+u_2)\frac{\partial}{\partial y^\varphi},
  \]
  and therefore $d\pi(F)=0$.
\end{example}

A particularly important class of vertical control systems is given by affine vertical control systems.

\begin{definition}
  A vertical lift control system is said to be \emph{affine} if it has the form
  \begin{equation}\label{eq:vertival affine system}
    \dot v = X_0^v(v) + \sum_{i=1}^m u_i X_i^v(v),
  \end{equation}
  where $X_0,\dots,X_m \in \mathfrak{X}(M)$.
\end{definition}

This class of systems is particularly relevant, as it describes control actions acting exclusively on the fiber variables through vertical lifts.

\begin{proposition}
  In induced coordinates $(x^i,y^i)$ on $TM$, the vertical affine control system is given by
  \[
    \dot x^j = 0, \qquad 
    \dot y^j = X_0^j(x) + \sum_{i=1}^m u_i X_i^j(x), \quad j=1,\dots,n.
  \]
\end{proposition}

\begin{proof}
  It is a direct application of Proposition \ref{prop:local_form_vertical}.
\end{proof}

We now present an example of an affine vertical control system on $TS^2$.
\begin{example}
  Let
  \[
    X_1=\frac{\partial}{\partial \theta},\qquad
    X_2=\frac{\partial}{\partial \varphi},\qquad
    X_0=\cos\varphi\,\frac{\partial}{\partial \theta}+\sin\theta\,\frac{\partial}{\partial \varphi}.
  \]
  Their vertical lifts are
  \[
    X_1^v=\frac{\partial}{\partial y^\theta},\qquad
    X_2^v=\frac{\partial}{\partial y^\varphi},\qquad
    X_0^v=\cos\varphi\,\frac{\partial}{\partial y^\theta}
    +\sin\theta\,\frac{\partial}{\partial y^\varphi}.
  \]
  Hence the affine vertical control system
  \[
    \dot v = X_0^v(v)+u_1X_1^v(v)+u_2X_2^v(v)
  \]
  takes the local form
  \[
    \dot\theta=0,\qquad \dot\varphi=0,
  \]
  \[
    \dot y^\theta=\cos\varphi+u_1,\qquad
    \dot y^\varphi=\sin\theta+u_2.
  \]
  Thus the dynamics acts only along the fibers of $TS^2\to S^2$.
\end{example}

We now describe some fundamental properties of the vertical control flow $\phi^v_t$.

\begin{proposition}\label{projectvertical}
  Let $\phi^v_t$ be the solution of a vertical affine control system associated with a control system $\phi_t$ on $M$. Then the following statements hold:
  \begin{enumerate}
    \item $\dot{\phi}^v_t$ is a vertical vector field;
    \item the trajectory $\phi^v_t$ projects onto a constant point $x_0 = \pi(\phi^v_0)$.
  \end{enumerate}
\end{proposition}

\begin{proof}
  Let $\phi^v_t$ be a solution of the vertical control system. We show that
  \[
    d\pi\big(\dot{\phi}^v_t\big)=0.
  \]
  Indeed,
  \[
    d\pi\big(\dot{\phi}^v_t\big)  = d\pi\!\left(X_0^v(\phi^v_t)+  \sum_{i=1}^{m} u_i(t)\, X_i^v(\phi^v_t) \right).
  \]
  Since each $X_i^v$ is a vertical vector field, we have
  \[
    d\pi\big(X_i^v(\phi^v_t)\big)=0,\quad i=0,1,\dots,m,
  \]
  and therefore
  \[
    d\pi\big(\dot{\phi}^v_t\big)=0.
  \]
  This proves that $\dot{\phi}^v_t$ is vertical.

  Moreover,
  \[
    \frac{d}{dt}\big(\pi(\phi^v_t)\big) = d\pi\big(\dot{\phi}^v_t\big)=0,
  \]
  hence $\pi(\phi^v_t)$ is constant. Taking $x_0=\pi(\phi^v_0)$, we conclude that
  \[
    \pi(\phi^v_t)=x_0 \quad \text{for all } t\geq 0.
  \]
\end{proof}

An explicit representation of solutions of affine vertical systems can be obtained by means of chronological calculus (see \cite[ch.~2]{agrachev}).

\begin{theorem}\label{thm:explicit_vertical_solution}
  Let $v_x \in T_xM$ and denote by $\phi^v(t,v_x,u)$ the corresponding trajectory of the affine vertical control system (\ref{eq:vertival affine system}). Then, for every $t\geq 0$, the solution admits the representation
  \begin{equation}\label{eq:chronological_representation}
    \phi^v(t,v_x,u) = \varphi_t^{v,0} \left(  \overrightarrow{\exp} \left(  \int_0^t \sum_{i=1}^m u_i(\tau)\, X_i^v\, d\tau \right)(v_x)\right),
  \end{equation}
  where $\varphi_t^{v,0}$ is the flow of $X_0^v$.
  Moreover, since the vertical lifts commute and the trajectory remains in the fiber $T_xM$, the solution can be written explicitly as
  \begin{equation}\label{eq:fiber_solution}
    v_x(t) = v_x(0) + t\,X_0(x) + \sum_{i=1}^m \left(\int_0^t u_i(\tau)\,d\tau\right) X_i(x) \;\in\; T_xM.
  \end{equation}
\end{theorem}

\begin{proof}
  Using chronological calculus (see \cite{agrachev}), the solution of the system can be written as
  \[
    \phi^v(t,v_x,u) = v_x \circ \overrightarrow{\exp} \int_0^t \left( X_0^v + \sum_{i=1}^m u_i(\tau)\, X_i^v \right)d\tau.
  \]
  Applying the variation formula yields the decomposition
  \[
    \phi^v(t,v_x,u) = v_x \circ \overrightarrow{\exp} \left(\int_0^t\left(\overrightarrow{\exp} \int_0^\tau \operatorname{ad}(X_0^v)\, d\theta \right)
    \sum_{i=1}^m u_i(\tau)\, X_i^v\, d\tau\right)
    \circ \overrightarrow{\exp} \int_0^t X_0^v\, d\tau.
  \]
  Since $X_0^v$ is autonomous, we have
  \[
    \overrightarrow{\exp}\int_0^\tau \operatorname{ad}(X_0^v)\, d\theta = e^{\tau\,\operatorname{ad}(X_0^v)}.
  \]
  Moreover, using $[X^v,Y^v]=0$ for all vector fields $X,Y$, it follows that
  \[
    \operatorname{ad}(X_0^v)X_i^v = 0, \qquad e^{\tau\,\operatorname{ad}(X_0^v)}X_i^v = X_i^v.
  \]
  Therefore,
  \[
    \phi^v(t,v_x,u) = v_x \circ \overrightarrow{\exp} \left(\int_0^t \sum_{i=1}^m u_i(\tau)\, X_i^v\, d\tau\right)
    \circ \overrightarrow{\exp} \int_0^t X_0^v\, d\tau,
  \]
  which gives \eqref{eq:chronological_representation}.
  Since the fields $X_i^v$ commute, the chronological exponential reduces to the ordinary exponential, yielding
  \[
    \overrightarrow{\exp}
    \left( \int_0^t \sum_{i=1}^m u_i(\tau)\, X_i^v\, d\tau \right) = \exp \left( \sum_{i=1}^m \left(\int_0^t u_i(\tau)\, d\tau\right) X_i^v\right).
  \]

  We now compute the action of these flows on the fiber $T_xM$. For each $i$, by definition of the vertical lift,
  \[
    X_i^v(w_x) = \left.\frac{d}{ds}\right|_{s=0}\big(w_x + s\,X_i(x)\big),
  \]
  so that
  \[
    \Phi_s^{X_i^v}(w_x)=w_x+s\,X_i(x).
  \]
  Thus,
  \[
    \exp \left(\sum_{i=1}^m \left(\int_0^t u_i(\tau)\,d\tau\right) X_i^v \right)(v_x)
    = v_x + \sum_{i=1}^m \left(\int_0^t u_i(\tau)\,d\tau\right) X_i(x).
  \]

  Similarly, the flow of $X_0^v$ is given by
  \[
    \varphi_t^{v,0}(w_x)=w_x+t\,X_0(x).
  \]
  Combining the two expressions, we obtain
  \[
    \phi^v(t,v_x,u)=v_x+t\,X_0(x)+\sum_{i=1}^m \left(\int_0^t u_i(\tau)\,d\tau\right) X_i(x),
  \]
  which proves \eqref{eq:fiber_solution}.
\end{proof}

The argument above is inspired by Proposition 3.1 of \cite{agrachev}.

To study controllability, we introduce the reachable sets. For $v_x \in TM$ and $t\geq 0$, define
\[
  \mathcal{R}^v_t(v_x)  = \left\{ w \in TM \;\middle|\; \phi^v(t,v_x,u)=w \text{ for some admissible control } u  \right\},
\]
and the full reachable set
\[
  \mathcal{R}^v(v_x)  = \bigcup_{t\geq 0} \mathcal{R}^v_t(v_x).
\]

\begin{corollary}
  Let $v_x \in TM$ with $\pi(v_x)=x$. Then, for all $t\geq 0$, the reachable set $\mathcal{R}^v_t(v_x)$ is contained in the fiber $T_xM$. In particular,
  \[
    \pi\big(\mathcal{R}^v_t(v_x)\big)=\{x\},
  \]
  and the same holds for $\mathcal{R}^v(v_x)$.
\end{corollary}

It follows immediately that the system is never controllable on the whole tangent bundle \(TM\). Indeed, if $v_x \in T_xM$ and $w_y \in T_yM$ with $x\neq y$, there is no admissible control $u$ and no time $t$ such that
\[
  \phi^v(t,v_x,u)=w_y,
\]
since trajectories remain in the initial fiber.

However, the system may still be controllable when restricted to each fiber $T_xM$, which motivates the study of fiberwise controllability.

\begin{theorem}\label{thm:fiber_controllability_vertical}
  Consider the affine vertical control system on $TM$ given by \eqref{eq:vertival affine system}. Assume that the admissible controls belong to $L^1([0,T],\mathbb R^m)$, so that for every  $(\alpha_1,\dots,\alpha_m)\in\mathbb R^m$ there exists an admissible control  $u=(u_1,\dots,u_m)$ satisfying
  \[
    \int_0^T u_i(s)\,ds=\alpha_i,
    \qquad i=1,\dots,m.
  \]
  Fix $x_0\in M$, $v_0\in T_{x_0}M$, and $T>0$. Then the following statements hold::
  \begin{enumerate}
    \item[(a)] The reachable set at time $T$ from $v_0$ is
    \[
      R_T^v(v_0) = v_0+T X_0(x_0)+\operatorname{span}\{X_1(x_0),\dots,X_m(x_0)\}.
    \]
    \item[(b)] The system is controllable on the fiber $T_{x_0}M$ if and only if
    \[
      \operatorname{span}\{X_1(x_0),\dots,X_m(x_0)\}=T_{x_0}M.
    \]
  \end{enumerate}
\end{theorem}

\begin{proof}
  By Theorem~\ref{thm:explicit_vertical_solution}, every trajectory starting from $v_0\in T_{x_0}M$ satisfies
  \[
    v(T)=v_0+T X_0(x_0)+\sum_{i=1}^m\left(\int_0^T u_i(s)\,ds\right)X_i(x_0).
  \]
  Hence every reachable point belongs to
  \[
    v_0+T X_0(x_0)+\operatorname{span}\{X_1(x_0),\dots,X_m(x_0)\},
  \]
  which proves one inclusion in (a).

  Conversely, let
  \[
    w=v_0+T X_0(x_0)+\sum_{i=1}^m\alpha_i X_i(x_0)
  \]
  be any point in
  \[
    v_0+T X_0(x_0)+\operatorname{span}\{X_1(x_0),\dots,X_m(x_0)\}.
  \]
    By the assumption on the admissible control class, there exists $u=(u_1,\dots,u_m)\in L^1([0,T],\mathbb R^m)$ such that
  \[
    \int_0^T u_i(s)\,ds=\alpha_i, \qquad i=1,\dots,m.
  \]
  Substituting into the explicit formula for the solution gives
  \[
    v(T)=w.
  \]
  Therefore,
  \[
    R_T^v(v_0)  = v_0+T X_0(x_0)+\operatorname{span}\{X_1(x_0),\dots,X_m(x_0)\},
  \]
  proving (a).

  Now (b) follows immediately from (a). Indeed, the reachable set fills the whole fiber $T_{x_0}M$ if and only if
  \[
    v_0+T X_0(x_0)+\operatorname{span}\{X_1(x_0),\dots,X_m(x_0)\}=T_{x_0}M.
  \]
  Since the left-hand side is an affine subspace of $T_{x_0}M$, this happens if and only if
  \[
    \operatorname{span}\{X_1(x_0),\dots,X_m(x_0)\}=T_{x_0}M.
  \]
\end{proof}

\begin{remark}
  The result shows that, although the dynamics is defined on the tangent bundle, the vertical structure reduces the controllability problem to a finite-dimensional linear control system on each fiber. This highlights a fundamental distinction between vertical and complete lifts: while complete lifts preserve the geometric complexity of the base dynamics, vertical lifts completely decouple the fibers and lead to purely algebraic controllability conditions.
\end{remark}

\section{Lifted Control Systems with Complete Drift and Vertical Controls} 

Let M be a smooth manifold and let $\pi: TM \to M$ be the canonical projection.  Let $Y, X_1, \dots, X_m \in \mathfrak{X}(M)$ be smooth vector fields, and consider the control system on $TM$
\begin{equation}\label{eq:complete_drift_vertical}
  \dot v(t) = Y^c\big(v(t)\big) + \sum_{i=1}^m u_i(t)\, X_i^v\big(v(t)\big), \qquad v(t) \in TM,
\end{equation}
where $Y^c$ denotes the complete lift of $Y$ and $X_i^v$ denotes the vertical lift of $X_i$. 
We denote by $\Phi^t$ the associated flow.

This system arises naturally from the affine dynamics on the base manifold $M$,
\[
  \dot x(t) = Y(x(t)) + \sum_{i=1}^m u_i(t)\, X_i(x(t)),
\]
when one is interested not only in the evolution of trajectories but also in the behavior of tangent directions along them. 
In particular, the system \eqref{eq:complete_drift_vertical} provides a geometric framework to describe how tangent vectors evolve along the flow of the drift field $Y$.

A key feature of this lifted system is that the base dynamics is entirely determined by $Y$. Indeed, if $v(t)$ is a solution of \eqref{eq:complete_drift_vertical} and $x(t)=\pi(v(t))$, then
\[
  \dot x(t) = Y(x(t)),
\]
so that the controls do not influence the trajectory on $M$. 

However, the system introduces an additional degree of freedom at the level of the tangent bundle. While the base trajectory is fixed, the tangent vector $v(t)\in T_{x(t)}M$ evolves according to a controlled dynamics. The complete lift $Y^c$ describes the transport of tangent vectors along the flow of $Y$, corresponding to the differential of the flow, whereas the vertical lifts $X_i^v$ act exclusively along the fibers, allowing one to modify the tangent component without altering the base point.

In this sense, the system does not control the trajectory itself, but rather the evolution of directions along it. This viewpoint provides a natural setting for the study of accessibility properties, Lie-algebraic generation, and lifted dynamics on the tangent bundle, establishing a direct link between the geometry of the base system and its induced dynamics on $TM$.

\begin{theorem}\label{thm:solution_system}
  Consider the lifted control system on $TM$ given by \eqref{eq:complete_drift_vertical}. Let $\varphi_t^Y$ denote the flow of $Y$. Then, for every admissible control $u \in L^1([0,T],\mathbb{R}^m)$, the solution at time $T$ is given by
  \begin{equation}\label{eq:endmap_formula}
    v(T) =  (d\phi_T^Y)_{x_0}(v_0) +\sum_{i=1}^m \int_0^T u_i(t)\,(\varphi_{T-t}^Y)_* X_i\big(\varphi_t^Y(x_0)\big)\,dt.
  \end{equation}
\end{theorem}

\begin{proof}
  Let $Z_t := Y^c + \sum_{i=1}^m u_i(t)\,X_i^v$ be the time-dependent vector field on $TM$. The solution of \eqref{eq:complete_drift_vertical} can be written using the chronological exponential as
  \[
    v(T) = \left(\overrightarrow{\exp}\int_0^T Z_t\,dt\right)(v_0).
  \]
  Using the variation-of-constants formula in chronological calculus, we obtain
  \[
    \overrightarrow{\exp}\int_0^T (Y^c + B_t)\,dt = e^{T Y^c} \circ \overrightarrow{\exp}\int_0^T (e^{-t Y^c})_* B_t\,dt,
  \]
  where $B_t = \sum_{i=1}^m u_i(t)\,X_i^v$. Hence,
  \[
    v(T) = e^{T Y^c} \left( \overrightarrow{\exp}\int_0^T \sum_{i=1}^m u_i(t)\,(e^{-t Y^c})_* X_i^v\,dt\right)(v_0).
  \]
  Since the flow of $Y^c$ is given by $e^{tY^c} = T\varphi_t^Y$, we have
  \[
    (e^{-t Y^c})_* X_i^v = (T\varphi_{-t}^Y)_* X_i^v = \big((\varphi_{-t}^Y)_* X_i\big)^v.
  \]
  Therefore,
  \[
    v(T) = T\varphi_T^Y \left( \overrightarrow{\exp}\int_0^T \sum_{i=1}^m u_i(t)\, \big((\varphi_{-t}^Y)_* X_i\big)^v\,dt\right)(v_0).
  \]
  Observe that the vector field inside the chronological exponential is vertical and, along the trajectory, acts along the fixed fiber $T_{x_0}M$. Since vertical vector fields of this form commute, the chronological exponential reduces to a translation:
  \[
    \left( \overrightarrow{\exp}\int_0^T \sum_{i=1}^m u_i(t)\,\big((\varphi_{-t}^Y)_* X_i\big)^v\,dt \right)(v_0)
    = v_0 + \int_0^T \sum_{i=1}^m u_i(t)\, (\varphi_{-t}^Y)_* X_i(x_0)\,dt.
  \]
  Applying $T\varphi_T^Y$ and using linearity along the fibers, we obtain
  \[
    v(T) = T\varphi_T^Y(v_0) + \int_0^T \sum_{i=1}^m u_i(t)\,T\varphi_T^Y\big((\varphi_{-t}^Y)_* X_i(x_0)\big)\,dt.
  \]
  Finally, using the identity
  \[
    T\varphi_T^Y \circ (\varphi_{-t}^Y)_* = (\varphi_{T-t}^Y)_*,
  \]
  we conclude that
  \[
    v(T) = T\varphi_T^Y(v_0) + \sum_{i=1}^m \int_0^T u_i(t)\, (\varphi_{T-t}^Y)_* X_i\big(\varphi_t^Y(x_0)\big)\,dt.
  \]
\end{proof}

\begin{example}
  Let $M=S^2$ and consider local spherical coordinates $(\theta,\varphi)$ on an open subset of $S^2$, with induced coordinates $(\theta,\varphi,y^\theta,y^\varphi)$ on $TS^2$. Take
  \[
    Y=\frac{\partial}{\partial \varphi},\qquad
    X_1=\frac{\partial}{\partial \theta},\qquad
    X_2=\frac{\partial}{\partial \varphi}.
  \]
  Since the coefficients of $Y$ are constant, its complete lift is
  \[
    Y^c=\frac{\partial}{\partial \varphi}.
  \]
  Moreover, the vertical lifts of $X_1$ and $X_2$ are
  \[
    X_1^v=\frac{\partial}{\partial y^\theta}, \qquad X_2^v=\frac{\partial}{\partial y^\varphi}.
  \]
  Therefore, the control system
  \[
    \dot v(t)=Y^c(v(t))+u_1(t)X_1^v(v(t))+u_2(t)X_2^v(v(t))
  \]
  takes the local form
  \[
    \dot\theta=0,\qquad \dot\varphi=1,
  \]
  \[
  \dot y^\theta=u_1(t),\qquad \dot y^\varphi=u_2(t).
  \]
  Hence the base trajectory is
  \[
  \theta(t)=\theta_0,\qquad \varphi(t)=\varphi_0+t,
  \]
  while the fiber variables satisfy
  \[
    y^\theta(t)=y^\theta(0)+\int_0^t u_1(s)\,ds,  \qquad  y^\varphi(t)=y^\varphi(0)+\int_0^t u_2(s)\,ds.
  \]
  Thus, the control acts only on the fiber coordinates, whereas the base point evolves according to the flow of $Y$ on $S^2$.
\end{example}

We now prove a condition for fiberwise controllability via transported vector fields.

\begin{theorem}\label{thm:fiber_transport_operator_strong}
  Consider the lifted control system \eqref{eq:complete_drift_vertical} on $TM$ with the controls belong to $L^1([0,T],\mathbb R^m)$. Fix $x_0\in M$, let $\varphi_t^Y$ be the flow of $Y$, and let $v_0\in T_{x_0}M$. For each $T>0$, define the linear operator
  \[
    L_T:L^1([0,T],\mathbb R^m)\to T_{\varphi_T^Y(x_0)}M
  \]
  by
  \[
    L_T(u):=\sum_{i=1}^m\int_0^T  u_i(t)\,(\varphi_{T-t}^Y)_*X_i\big(\varphi_t^Y(x_0)\big)\,dt.
  \]
  Then the following statements hold:
  \begin{enumerate}
    \item[(a)] The reachable set at time $T$ from $v_0$ is
    \[
      R_T(v_0)  = (d\varphi_T^Y)_{x_0}(v_0)+\operatorname{Im}(L_T).
    \]
    \item[(b)] One has
    \[
      \operatorname{Im}(L_T)  = (d\varphi_T^Y)_{x_0}  \left(  \operatorname{span} \left\{ (\varphi_{-t}^Y)_*X_i(x_0):\ t\in[0,T],\ i=1,\dots,m\right\}\right).
    \]
    \item[(c)] The system is controllable along the fiber $T_{\varphi_T^Y(x_0)}M$ at time $T$ if and only if
    \[
      \operatorname{span}
      \left\{ (\varphi_{-t}^Y)_*X_i(x_0):\ t\in[0,T],\ i=1,\dots,m  \right\}  = T_{x_0}M.
    \]
    \item[(d)] In particular, if
    \[
      \operatorname{span} \big\{  \operatorname{ad}_Y^kX_i(x_0):\ k\ge0,\ i=1,\dots,m \big\}  = T_{x_0}M,
    \]
    then the system is controllable along the fiber $T_{\varphi_T^Y(x_0)}M$ at time $T$.
  \end{enumerate}
\end{theorem}

\begin{proof}
  We divide the proof into several steps.

  \medskip
  \noindent
  \textbf{Step 1:}
  By the representation formula established for the lifted system,  for every admissible control $u\in L^1([0,T],\mathbb R^m)$ the corresponding solution $v_u(\cdot)$ of \eqref{eq:complete_drift_vertical} with initial condition $v_u(0)=v_0\in T_{x_0}M$ satisfies
  \begin{equation}\label{eq:explicit_solution_strong}
    v_u(T)  = (d\varphi_T^Y)_{x_0}(v_0) + \sum_{i=1}^m
    \int_0^T  u_i(t)\,(\varphi_{T-t}^Y)_*X_i\big(\varphi_t^Y(x_0)\big)\,dt.
  \end{equation}
  Therefore,
  \[
    v_u(T)=(d\varphi_T^Y)_{x_0}(v_0)+L_T(u).
  \]
  Since the reachable set at time $T$ is, by definition,
  \[
    R_T(v_0)=\{v_u(T):u\in L^1([0,T],\mathbb R^m)\},
  \]
  it follows immediately that
  \[
    R_T(v_0)  = (d\varphi_T^Y)_{x_0}(v_0)+\operatorname{Im}(L_T).
  \]
  This proves (a).

  \medskip
  \noindent
  \textbf{Step 2:}
  For each $t\in[0,T]$ and $i\in\{1,\dots,m\}$, using the chain rule for the flow,  we have
  \[
    (\varphi_{T-t}^Y)_*X_i\big(\varphi_t^Y(x_0)\big)  = (d\varphi_T^Y)_{x_0}\Big((\varphi_{-t}^Y)_*X_i(x_0)\Big).
  \]
  Indeed,
  \[
    \varphi_T^Y=\varphi_{T-t}^Y\circ\varphi_t^Y,
  \]
  hence
  \[
    (d\varphi_T^Y)_{x_0}  = (d\varphi_{T-t}^Y)_{\varphi_t^Y(x_0)} \circ (d\varphi_t^Y)_{x_0},
  \]
  and therefore
  \[
    (d\varphi_T^Y)_{x_0}\big((d\varphi_{-t}^Y)_{\varphi_t^Y(x_0)}X_i(\varphi_t^Y(x_0))\big) =
    (d\varphi_{T-t}^Y)_{\varphi_t^Y(x_0)}X_i(\varphi_t^Y(x_0)).
  \]
  This is precisely the desired identity. Substituting into the definition of $L_T$, we obtain
  \begin{align}
    L_T(u)
    &=
    \sum_{i=1}^m\int_0^T u_i(t)\,(d\varphi_T^Y)_{x_0}\big((\varphi_{-t}^Y)_*X_i(x_0)\big)\,dt \notag\\
    &=  
    (d\varphi_T^Y)_{x_0}
    \left(  \sum_{i=1}^m\int_0^T  u_i(t)\,(\varphi_{-t}^Y)_*X_i(x_0)\,dt\right).
    \label{eq:LT_factorization}
  \end{align}
  Define
  \[
    S_T:= \operatorname{span} \left\{ (\varphi_{-t}^Y)_*X_i(x_0):\ t\in[0,T],\ i=1,\dots,m  \right\}  \subset T_{x_0}M.
  \]
  Then \eqref{eq:LT_factorization} shows that
  \[
    \operatorname{Im}(L_T)\subset (d\varphi_T^Y)_{x_0}(S_T).
  \]
  Thus, one inclusion in (b) is immediate.

  \medskip
  \noindent
  \textbf{Step 3:}
  We now prove that every generator
  \[
    (d\varphi_T^Y)_{x_0}\big((\varphi_{-t_0}^Y)_*X_j(x_0)\big)
  \]
  belongs to $\operatorname{Im}(L_T)$, for every fixed $t_0\in[0,T]$  and $j\in\{1,\dots,m\}$. Fix such $t_0$ and $j$. For $\varepsilon>0$ small enough such that $t_0+\varepsilon\le T$, define the control  $u^\varepsilon=(u_1^\varepsilon,\dots,u_m^\varepsilon)$ by
  \[
    u_i^\varepsilon(t)=
    \begin{cases}
      \varepsilon^{-1}, & \text{if } i=j \text{ and } t\in[t_0,t_0+\varepsilon],\\[1mm]
      0, & \text{otherwise}.
    \end{cases}
  \]
  Then $u^\varepsilon\in L^1([0,T],\mathbb R^m)$ and, by \eqref{eq:LT_factorization},
  \[
    L_T(u^\varepsilon)  = (d\varphi_T^Y)_{x_0}
    \left(
    \frac1\varepsilon \int_{t_0}^{t_0+\varepsilon}  (\varphi_{-t}^Y)_*X_j(x_0)\,dt  \right).
  \]
  Since the map
  \[
    t\longmapsto (\varphi_{-t}^Y)_*X_j(x_0)\in T_{x_0}M
  \]
  is continuous, it follows that
  \[
    \frac1\varepsilon \int_{t_0}^{t_0+\varepsilon}  (\varphi_{-t}^Y)_*X_j(x_0)\,dt  \longrightarrow (\varphi_{-t_0}^Y)_*X_j(x_0)  \qquad\text{as }\varepsilon\to0^+.
  \]
  Hence
  \[
    L_T(u^\varepsilon)\longrightarrow (d\varphi_T^Y)_{x_0}\big((\varphi_{-t_0}^Y)_*X_j(x_0)\big).
  \]
  Now $\operatorname{Im}(L_T)$ is a linear subspace of the finite-dimensional vector space  $T_{\varphi_T^Y(x_0)}M$; therefore it is closed. Since each $L_T(u^\varepsilon)$ belongs to $\operatorname{Im}(L_T)$, the limit also belongs to $\operatorname{Im}(L_T)$:
  \[
    (d\varphi_T^Y)_{x_0}\big((\varphi_{-t_0}^Y)_*X_j(x_0)\big)\in \operatorname{Im}(L_T).
  \]
  Because $t_0$ and $j$ are arbitrary, we conclude that
  \[
    (d\varphi_T^Y)_{x_0}(S_T)\subset \operatorname{Im}(L_T).
  \]
  Combining this with the opposite inclusion proved in Step 2 gives
  \[
    \operatorname{Im}(L_T)=(d\varphi_T^Y)_{x_0}(S_T),
  \]
  which proves (b).

  \medskip
  \noindent
  \textbf{Step 4:}
  By (a) and (b),
  \[
    R_T(v_0)  = (d\varphi_T^Y)_{x_0}(v_0)+(d\varphi_T^Y)_{x_0}(S_T).
  \]
  Since $(d\varphi_T^Y)_{x_0}:T_{x_0}M\to T_{\varphi_T^Y(x_0)}M$ is a linear isomorphism, the affine set above coincides with the whole fiber $T_{\varphi_T^Y(x_0)}M$ if and only if
  \[
    S_T=T_{x_0}M.
  \]
  By the definition of $S_T$, this means precisely
  \[
    \operatorname{span} \left\{ (\varphi_{-t}^Y)_*X_i(x_0):\ t\in[0,T],\ i=1,\dots,m  \right\}
    = T_{x_0}M.
  \]
  This proves (c).

  \medskip
  \noindent
  \textbf{Step 5:}
  Assume now that
  \begin{equation}\label{eq:ad_condition_strong}
    \operatorname{span}\big\{\operatorname{ad}_Y^kX_i(x_0):k\ge0,\ i=1,\dots,m\big\}=T_{x_0}M.
  \end{equation}
  We prove that then $S_T=T_{x_0}M$.

  Suppose by contradiction that $S_T\neq T_{x_0}M$. Since $S_T$ is a proper subspace of the finite-dimensional space $T_{x_0}M$,  there exists a nonzero covector $\lambda\in T_{x_0}^*M$ such that
  \[
    \lambda(w)=0 \qquad\text{for all } w\in S_T.
  \]
  In particular, for every $i=1,\dots,m$ and every $t\in[0,T]$,
  \[
    \lambda\big((\varphi_{-t}^Y)_*X_i(x_0)\big)=0.
  \]
  For each fixed $i$, define the smooth real-valued function
  \[
    f_i(t):=\lambda\big((\varphi_{-t}^Y)_*X_i(x_0)\big), \qquad t\in[0,T].
  \]
  Then $f_i(t)\equiv0$ on $[0,T]$, hence
  \[
    f_i^{(k)}(0)=0 \qquad\text{for all } k\ge0.
  \]
  On the other hand, differentiating the family
  \[
    Z_i(t):=(\varphi_{-t}^Y)_*X_i
  \]
  with respect to $t$, we obtain
  \[
    \frac{d}{dt}Z_i(t)=-(\varphi_{-t}^Y)_*[Y,X_i].
  \]
  Iterating this identity yields
  \[
    \frac{d^k}{dt^k}Z_i(t)=(-1)^k(\varphi_{-t}^Y)_*\big(\operatorname{ad}_Y^kX_i\big),
  \]
  and therefore, evaluating at $t=0$,
  \[
    \frac{d^k}{dt^k}Z_i(0)=(-1)^k\operatorname{ad}_Y^kX_i.
  \]
  Consequently,
  \[
    f_i^{(k)}(0)  = \lambda\!\left(\frac{d^k}{dt^k}Z_i(0)(x_0)\right)
    = (-1)^k\lambda\big(\operatorname{ad}_Y^kX_i(x_0)\big).
  \]
  Since $f_i^{(k)}(0)=0$, we obtain
  \[
    \lambda\big(\operatorname{ad}_Y^kX_i(x_0)\big)=0
    \qquad\text{for all } k\ge0,\ i=1,\dots,m.
  \]
  Thus $\lambda$ annihilates the subspace
  \[
    \operatorname{span}\big\{\operatorname{ad}_Y^kX_i(x_0):k\ge0,\ i=1,\dots,m\big\}.
  \]
  By \eqref{eq:ad_condition_strong}, this subspace is all of $T_{x_0}M$, so $\lambda$ must vanish on $T_{x_0}M$, contradicting $\lambda\neq0$.  Hence $S_T=T_{x_0}M$. By part (c), the system is controllable along the fiber $T_{\varphi_T^Y(x_0)}M$ at time $T$. This proves (d).
\end{proof}

We present a controllable example on \(TS^2\).

\begin{example}
  Let $M=S^2$ and consider a local spherical chart $(\theta,\varphi)$, with induced coordinates $(\theta,\varphi,y^\theta,y^\varphi)$ on $TS^2$. Take
  \[
    Y=\frac{\partial}{\partial\varphi}, \qquad 
    X_1=\frac{\partial}{\partial\theta}, \qquad 
    X_2=\frac{\partial}{\partial\varphi}.
  \]
  Then the lifted system
  \[
    \dot v(t)=Y^c(v(t))+u_1(t)X_1^v(v(t))+u_2(t)X_2^v(v(t))
  \]
  is of the form \eqref{eq:complete_drift_vertical}. The flow of $Y$ is given by
  \[
    \varphi_t^Y(\theta_0,\varphi_0)=(\theta_0,\varphi_0+t).
  \]
  Since
  \[
    [Y,X_1]=0, \qquad [Y,X_2]=0,
  \]
    the vector fields $X_1$ and $X_2$ are invariant under the flow of $Y$, that is,
  \[
    (\varphi_{-t}^Y)_*X_i(x_0)=X_i(x_0), \qquad \text{for all } t\in[0,T], \ i=1,2.
  \]
  It follows that the subspace $S_T$ introduced in Theorem~\ref{thm:fiber_transport_operator_strong} reduces to
  \[
    S_T = \operatorname{span}\{X_1(x_0),X_2(x_0)\}.
  \]
  Since $\{X_1,X_2\}=\{\partial/\partial\theta,\partial/\partial\varphi\}$ is a basis of $T_{x_0}S^2$, we obtain
  \[
    S_T=T_{x_0}S^2.
  \]
  By Theorem~\ref{thm:fiber_transport_operator_strong}, this implies
  \[
    \operatorname{Im}L_T  = (d\varphi_T^Y)_{x_0}(T_{x_0}S^2)  = T_{\varphi_T^Y(x_0)}S^2.
  \]
  Consequently,
  \[
    \mathcal R_T(v_0) = T_{\varphi_T^Y(x_0)}S^2,
  \]
  and the system is controllable along the fiber at time $T$.

  \medskip

  \noindent
  In this example, controllability follows from the fact that the control directions are invariant along the flow of $Y$, so no new directions are generated by transport. This represents the simplest situation covered by Theorem~\ref{thm:fiber_transport_operator_strong}, where fiberwise controllability reduces to a pointwise rank condition.
\end{example}

\begin{example}
  Let $M=\mathbb R^2$ with global coordinates $(x,y)$, and induced coordinates  $(x,y,v_x,v_y)$ on $T\mathbb R^2$. Consider
  \[
    Y=x\frac{\partial}{\partial y},
    \qquad
    X_1=\frac{\partial}{\partial x}.
  \]
  Then
  \[
    [Y,X_1] = \left[x\frac{\partial}{\partial y},\frac{\partial}{\partial x}\right]
    = -\frac{\partial}{\partial y}\neq 0,
  \]
  so the control direction is not invariant under the flow of $Y$.  The flow of $Y$ is
  \[
    \varphi_t^Y(x_0,y_0)=(x_0,y_0+t\,x_0).
  \]
  Its differential is
  \[
    (d\varphi_t^Y)_{(x_0,y_0)}  =
    \begin{pmatrix}
      1 & 0\\
      t & 1
    \end{pmatrix},
  \]
  hence
  \[
    (\varphi_{-t}^Y)_*X_1(x_0,y_0)  = (d\varphi_{-t}^Y)_{\varphi_t^Y(x_0,y_0)}  
    \left(\frac{\partial}{\partial x}\right)
    =
    \frac{\partial}{\partial x} - t\,\frac{\partial}{\partial y}.
  \]
  Therefore, for every $T>0$,
  \[
    S_T = \operatorname{span}
    \left\{
    (\varphi_{-t}^Y)_*X_1(x_0,y_0):\, t\in[0,T]
    \right\}
    =
    \operatorname{span}
    \left\{
    \frac{\partial}{\partial x}-t\frac{\partial}{\partial y}: t\in[0,T]
    \right\}.
  \]
  Since the family above contains, for instance, the vectors
  \[
    \frac{\partial}{\partial x} \qquad\text{and}\qquad  \frac{\partial}{\partial x}-T\frac{\partial}{\partial y},
  \]
  it follows that
  \[
    S_T=T_{(x_0,y_0)}\mathbb R^2.
  \]
  Hence, by Theorem~\ref{thm:fiber_transport_operator_strong},
  \[
    \operatorname{Im}L_T= (d\varphi_T^Y)_{(x_0,y_0)}\big(T_{(x_0,y_0)}\mathbb R^2\big)
    = T_{\varphi_T^Y(x_0,y_0)}\mathbb R^2,
  \]
  and consequently
  \[
    \mathcal R_T(v_0) = T_{\varphi_T^Y(x_0,y_0)}\mathbb R^2.
  \]
  Thus the lifted system is controllable along the fiber at time $T$ for every $T>0$. In local coordinates on $T\mathbb R^2$, the lifted system is
  \[
    \dot x=0,\qquad
    \dot y=x,\qquad
    \dot v_x=u,\qquad
    \dot v_y=v_x.
  \]
  This example shows that, although there is only one control vector field on the base, its transport along the drift generates a second independent direction. Hence fiberwise controllability is achieved through the interaction between the complete drift and the vertical control.
\end{example}

\begin{remark}
  The previous example shows that fiberwise controllability may hold even when the  original control family $\{X_1,\dots,X_m\}$ does not span the tangent space at the  initial point. What matters is the span of the transported directions $(\varphi_{-t}^Y)_*X_i(x_0)$ along the drift flow.
\end{remark}

\end{document}